\documentclass{amsart}
\usepackage{amsthm,amsfonts,amsmath,amscd,amssymb,latexsym,epsfig}
\usepackage{tikz-cd}

\newcommand{\cx}{{\mathbb C}}

\newcommand{\Res}{\operatorname{Res}}

\newcommand{\Ker}{\operatorname{Ker}}

\newcommand{\Rat}{\operatorname{Rat}}

\newcommand{\ol}{\overline}

\numberwithin{equation}{section}

\newtheorem{theorem}{Theorem}[section]

\newtheorem{lemma}[theorem]{Lemma}

\newtheorem{proposition}[theorem]{Proposition}

\newcounter{mtheorem}

\newtheorem{mtheorem}[mtheorem]{Theorem}

\theoremstyle{remark}

\newtheorem{remark}[theorem]{Remark}

\newtheorem{definition}[theorem]{Definition}

\newtheorem{example}[theorem]{Example}

\newcommand{\oN}{{\mathbb{N}}}

\newcommand{\oP}{{\mathbb{P}}}

\newcommand{\oR}{{\mathbb{R}}}

\newcommand{\oZ}{{\mathbb{Z}}}

\newcommand{\sF}{{\mathcal{F}}}
\newcommand{\sN}{{\mathcal{N}}}
\newcommand{\sO}{{\mathcal{O}}}

\begin{document}

\title{Hyperk\"ahler geometry of rational curves in twistor spaces}
\author{Roger Bielawski \and Naizhen Zhang}
\address{Institut f\"ur Differentialgeometrie,
Leibniz Universit\"at Hannover,
Welfengarten 1, 30167 Hannover, Germany}

%\thanks{}
%\email{R.Bielawski@ed.ac.uk}
%\date{\today}

%\thanks{}
%\dedicatory{}

\begin{abstract} We investigate the pseudo-hyperk\"ahler geometry of higher degree rational curves in the twistor space of a  hyperk\"ahler $4$-manifold. \end{abstract}

\maketitle

\thispagestyle{empty}

%\section{}

The twistor space of a hypercomplex or a hyperk\"ahler manifold $M$ is a complex manifold $Z$ equipped with a holomorphic submersion $\pi:Z\to \oP^1$ and an antiholomorphic involution $\sigma$ covering the antipodal map. The manifold $M$ is then recovered as (a component of) the Kodaira moduli space of $\sigma$-invariant sections of $\pi$ with normal bundle splitting as $\bigoplus\sO(1)$.
In \cite{Sigma} the first author observed that, if $\dim M=4$, then we also obtain a hypercomplex or pseudo-hyperk\"ahler structure on a subset of the Douady space consisting of $\sigma$-invariant curves of degree $d$, $d>1$, which are  cohomologically stable, i.e.\  satisfy $h^1(N_{C/Z}(-2))=0$.   
In \cite{BP2} C. Peternell and the first author showed that in the case of curves of genus $0$ in $\oP^3\backslash \oP^1$ (i.e.\ in the twistor space of the flat $\oR^4$) this pseudo-hyperk\"ahler structure can be obtained as a pseudo-hyperk\"ahler quotient of a flat space by a non-reductive Lie group. Even  in that case, however, we had been unable to determine the signature of the metric for $d>3$.
\par
In this work we investigate the pseudo-hyperk\"ahler geometry of higher degree $\oP^1$'s embedded  in the twistor space of an arbitrary $4$-dimensional hyperk\"ahler manifold. First of all, if such a $\oP^1$ of degree $d$ is to satisfy reality conditions, then $d$ must be odd. This has been proved in \cite[Prop.\ 5.9]{BP2}, but it also follows from the observation that a rational map $\phi:\oP^1\to\oP^1$ of degree $d$ can commute with the antipodal map only for odd $d$. With this restriction, let us denote by $M_d$ the subset of the Douady space consisting of $\sigma$-invariant cohomologically stable $\oP^1\subset Z$ of degree $d$. $M_d$ is hypercomplex, resp.\ pseudo-hyperk\"ahler, if $M$ is hypercomplex, resp. hyperk\"ahler. We remark that in this situation ``cohomologically stable" is equivalent to $N_{C/Z}\simeq \sO_{\oP^1}(2d-1)\oplus  \sO_{\oP^1}(2d-1)$.
Our main result is:
\begin{mtheorem} Let $M$ be a $4$-dimensional hyperk\"ahler  manifold. Suppose that $d\in \oN$ is odd and $M_d$ is nonempty. Then:
\begin{itemize}
\item[(i)] the signature of the pseudo-hyperk\"ahler metric on $M_d$ is $(2d+2,2d-2)$;
\item[(ii)] there exists a natural submersion $\rho: M_d\to \oR\oP^{2d-2}$ and an open dense subset $U$ of $M_d$ such
each fibre of $\rho|_U$ has a natural $d$-hypercomplex\footnote{A definition of a $d$-hypercomplex manifold may be found in \S\ref{realm}.} structure.
\end{itemize}
\end{mtheorem}
Part (ii) holds also if $M$ is only hypercomplex. The subset $U$ consists of $C$ such that the restricted vertical tangent bundle $(\Ker d\pi)|_C$ is isomorphic to $ \sO_{\oP^1}(d)\oplus \sO_{\oP^1}(d)$.
\par
The structure of the paper is as follows. In the next section we recall facts about the geometry of degree $d$ curves (of arbitrary arithmetic genus) in $Z$. We also interpret the $\sO_{\oP^1}(2)$-valued symplectic form on the fibres of 
the twistor space of such curves directly in terms of the normal bundles of curves (as long as they are local complete intersections). In \S 2 we define and study the map $\rho$ without any reality assumptions. These are imposed in \S 3, where we prove Theorem A. Finally, we discuss in detail the case of degree $d$ $\oP^1$'s embedded in the twistor space of an ALE or ALF gravitational instanton of type $A_k$. In the ALE case we can actually view $M_d$ as an open subset of the real locus of the Hilbert scheme of degree $d$ rational curves on a singular Fano $3$-fold - a hypersurface in a weighted projective $4$-space (cf.\  Remark \ref{Fano}).

                                                                                                                                                                                                                                                               \vspace{1mm}
                                                                                                                                                                                                                                                                
{\em Acknowledgement.}  This work has been carried out while both authors were members of, and the second author was fully funded by the DFG Priority Programme 2026 ``Geometry at infinity", the support of which is gratefully acknowledged.

\section{Geometry of Douady spaces of curves in a twistor space\label{Douady}}

Let $Z$ be a complex $3$-dimensional manifold with a holomorphic submersion $\pi:Z\to  \oP^1$. We write $\sO_Z(i)$ for $\pi^\ast \sO_{\oP^1}(i)$ and $\sF(i)$ for $\sF\otimes\sO_Z(i)$ for any sheaf $\sF$ on $Z$. We denote by $T_F$ the vertical tangent bundle $\Ker d\pi$  of $Z$. From the exact sequence
$$0\longrightarrow T_F\longrightarrow TZ \longrightarrow \pi^\ast T\oP^1\longrightarrow 0, $$
we conclude that $K_Z\simeq \Lambda^2 T^\ast_F(-2)$. In particular, an $\sO(2)$-valued symplectic form $\omega$ along the fibres of $\pi$, i.e.\ a trivialisation of $\Lambda^2 T^\ast_F(2)$, can be viewed as a nowhere  vanishing section of $K_Z(4)$.
\par
We now consider the subset $X_{d}$  of the Douady space of $1$-dimensional compact subspaces of $Z$ consisting of subschemes $C$ such that
 $\pi|_C:C\to \oP^1$ is flat of degree $d$. In particular, each such $C$ is pure-dimensional and Cohen-Macaulay. We denote by $X_{d}^{(i)}$, $i=0,1,2$, the subset of $X_{d}$ consisting of $C$, the normal sheaf $\sN_{C/Z}$ of which satisfies $h^1(\sN(-i))=0$. We summarize the main properties of $X_{d}^{(i)}$ as follows:
 %%%%
 \begin{proposition} In each statement below suppose that the corresponding $X_{d}^{(i)}$, $i=0,1,2$, is nonempty.
 \begin{itemize}
 \item[(i)]$X_{d}^{(0)}$  a smooth $4d$-dimensional manifold with a canonical isomorphism\\ $T_C X_{d}^{(0)}\simeq H^0(C,\sN_{C/Z})$ for each $C$.
 \item[(ii)] $X_{d}^{(1)}$ is equipped with a natural integrable $2$-Kronecker structure, i.e.\ a holomorphic vector bundle $E$, $E_C=  H^0(C,\sN_{C/Z}(-1))$, and a bundle map $\alpha:E\otimes \cx^2\to TX_{d}^{(1)}$ such that  $\alpha(E\otimes v)$ is an integrable rank $2d$ distribution for any nonzero $v\in \cx^2$.
\item[(iii)] $X_{d}^{(2)}$ is a $\cx$-hypercomplex manifold, i.e.\ the map $\alpha$ is an isomorphism everywhere. Consequently $X_{d}^{(2)}$ is equipped with a holomorphic Obata connection, i.e. a torsion-free holomorphic connection with holonomy in $GL(d,\cx)\simeq GL(E)$.
\item[(iv)] If $Z$ is also equipped with  an $\sO(2)$-valued symplectic form  along the fibres of $\pi$, then $X_{d}^{(2)}$ is a
$\cx$-hyperk\"ahler manifold, i.e.\ it has a nowhere degenerate  $\cx$-valued symmetric bilinear form $g$, such that the corresponding holomorphic Levi-Civita connection coincides with the Obata connection.
\end{itemize}\label{geom}
\end{proposition}
\begin{proof} Part (i) is easy in the case when $Z$ is quasiprojective. It follows then from  the fact that codimension $2$ Cohen-Macualay subspaces are locally unobstructed \cite[\S 2.8]{Hart}. In the general case we have to proceed differently. We consider the relative Hilbert scheme $Z^{[d]}_\pi$ of $d$ points along the fibres of $\pi$. It is a smoooth $(2d+1)$-dimensional manifold with a holomorphic submersion $\pi^{[d]}:Z^{[d]}_\pi \to \oP^1$, and the same argument as in \cite[Prop.\ 3.1]{BP2} shows that $X_d$ is isomorphic to the Douady space of sections of $\pi^{[d]}$. Furthermore, \cite[Lemma 3.2]{BP2} remains true, so that the normal bundle $N_s$ of a section $s$ corresponding to a curve $C$ is isomorphic to $\pi_\ast \sN_{C/Z}$. Hence, if $H^1(C,\sN_{C/Z})=0$, then $h^1(N_s)=0$, and this means that the Douady space of sections is smooth at $s$. This proves (i).
Parts (ii)-(iv) have been proved in \cite{Sigma, BP1}.
\end{proof}

In the case when the curve $C$ is a {\em local complete intersection} (lci), we can say more:
%%%
\begin{proposition} \begin{itemize}
\item[(i)] If $C\in X_{d}^{(0)}$, resp.  $C\in X_{d}^{(1)}$, is lci, then there is a canonical isomorphism $$T_C^\ast X_{d}^{(0)}\simeq H^1(C, \sN_{C/Z}\otimes K_Z),\quad \text{resp.}\enskip E_C^\ast \simeq H^1(C,\sN_{C/Z}\otimes K_Z(1)).$$
\item[(ii)] If $C\in X_{d}^{(2)}$ is lci, then there are additional canonical isomorphisms 
$$ E_C\simeq H^1(C,\sN_{C/Z}(-3)),\quad E_C^\ast\simeq H^0(C,\sN_{C/Z}\otimes K_Z(3)).$$
\end{itemize}
\end{proposition}
\begin{proof} Write $N$ for the locally free sheaf $\sN_{C/Z}$. The adjunction formula holds for lci subschemes \cite[Ch.\ 6, Thm.\ 4.9]{QL},
and hence:
%is lci, we have (see, e.g., \cite[Ch.III, Thm.\ 7.11]{Hart0} where $C$ is embedded in a projective space, but the proof is easily adapted to any smooth $Z$)
\begin{equation}  K_C\simeq \left.{K_Z}\right|_C\otimes \Lambda^2 N.\label{KCZ}\end{equation}
%\begin{equation} K_C\otimes N^\ast\simeq N\otimes K_Z.\label{KCZ}\end{equation}
Consequently:
$$ T_C^\ast X_{d}^{(0)}\simeq  H^0(C,N)^\ast\simeq H^1(C,K_C\otimes N^\ast)\simeq H^1(C,K_Z\otimes N).$$
The second isomorphism in (i) follows completely analogously, given that $E_C\simeq H^0(C,N(-1))$.
\par
For (ii) observe that we have a short exact sequence
$$ 0\longrightarrow N(-3)\longrightarrow N(-2)\oplus N(-2)\longrightarrow N(-1)\longrightarrow 0, $$ 
from which the first isomorphism follows immediately, since $N(-2)$ has trivial cohomology. Since $N^\ast\otimes K_C(2)$ also has trivial cohomology, the same argument, using the exact sequence
$$ 0\longrightarrow N^\ast\otimes  K_C(1)\longrightarrow \bigl(N^\ast\otimes K_C(2)\bigr)^{\oplus 2}\longrightarrow N^\ast\otimes K_C(3)\longrightarrow 0$$
and \eqref{KCZ}, shows the second isomorphism.
\end{proof}

It follows that a nowhere vanishing section $\omega$ of $\Lambda^2 T_F^\ast(2)\simeq K_Z(4)$ defines an isomorphism
\begin{equation}E_C\simeq H^0(C,\sN_{C/Z}(-1))\stackrel{\cdot\omega}{\longrightarrow}  H^0(C,\sN_{C/Z}\otimes K_Z(3))\simeq E_C^\ast\label{omegad}\end{equation}
for any lci curve $C\in X_{d}^{(2)}$. Write $\omega^{[d]}$ for the corresponding nondegenerate bilinear form on $E$ given by $(s,t)\mapsto (s\omega)(t)$.
\begin{remark} This construction of a symplectic form on $E$ is due to Nash \cite{N}, who showed that the hyperk\"ahler structure on a moduli space of framed Euclidean $SU(2)$-monopoles can be obtained this way.\end{remark}
\par
Let $\zeta\in\oP^1$. For a $C\in  X_{d}^{(2)}$, sections of $\sN_{C/Z}(-1)$ can be identified with the tangent space
to $C_\zeta=C\cap \pi^{-1}(\zeta)$ in the Hilbert scheme of $d$ points in the fibre $\pi^{-1}(\zeta)$. Formula \eqref{omegad} implies that $\omega^{[d]}$ coincides with the induced symplectic form \cite{Beau} on the Hilbert scheme of points (this is obvious on the subset where $C_\zeta$ consists of distinct points, and hence, by continuity, everywhere).
Therefore (cf.\ \cite{Sigma}) the symplectic form $\omega^{[d]}$ induces a $\cx$-hyperk\"ahler structure on  $X_{d}^{(2)}$, and a pseudo-hyperk\"ahler structure on the $\sigma$-invariant subset of $X_{d}^{(2)}$, if $Z$ is equipped with an antiholomorphic involution $\sigma$ covering the antipodal map.
\par
Following Nash \cite{N}, we are going to give another proof of the skew-symmetry of $\omega^{[d]}$, since the argument will be helpful when proving Theorem A(i).
%%%%
\begin{proposition} $\omega^{[d]}$ is skew-symmetric.\label{other}\end{proposition}
\begin{proof}\let\qed\relax We can express $\omega^{[d]}$ as the composition of the natural skew-symmetric map 
 $$H^0(C,N(-1))\times H^0(C,N(-1)){\to}H^0\bigl(C,(\Lambda^2N)(-2)\bigr)$$
 with 
 $$ H^0\bigl(C,(\Lambda^2N)(-2)\bigr)\simeq H^0(C,K_Z^\ast(-2)\otimes K_C)\stackrel{\cdot(\omega\lambda)}{\longrightarrow} H^1(C,K_C)\simeq \cx,$$
 where $\lambda\in H^1(C,\sO_C(-2))$ is the pullback of the extension class of
 $$0\to \sO_{\oP^1}(-3)\to \sO_{\oP^1}(-2)\oplus \sO_{\oP^1}(-2)\to \sO_{\oP^1}(-1)\to 0.\quad\quad\Box$$
 \end{proof}
%%%

 \section{Rational curves}
 
Let $\pi:Z\to \oP^1$ be as in the previous section. We denote by $X_{d,0}$ the component of $X_d$ consisting of smooth rational curves $C\simeq \oP^1$, and write $X_{d,0}^{(i)}=X_d^{(i)}\cap X_{d,0}$, $i=0,1,2$. We remark that $C\in X_{d,0}^{(2)}$ if and only if its normal bundle is isomorphic to $\sO_{\oP^1}(2d-1)\oplus \sO_{\oP^1}(2d-1)$.
\par
For a $C\in X_{d,0}$, let $R_C$, resp.\ $B_C$, be the ramification divisor, resp.\ the branch divisor, of $\pi|_C$. These are $0$-dimensional subschemes of $\oP^1$ and we obtain a holomorphic map
\begin{equation} \rho:  X_{d,0}\longrightarrow \oP^{2d-2}, \quad C\mapsto B_C.\label{rho}\end{equation}
%%%
\subsection{Covers  of $\oP^1$ and their parametrisations\label{covers}}
In order to understand the map $\rho$, we make a brief detour. The map $\rho$ can be viewed abstractly as associating to a degree cover $\pi:C\to \oP^1$ its branch divisor $B_C$. On the other hand, we can also parameterise $C$, $f:\oP^1\to C$, and obtain a degree $d$ rational map $\phi=\pi\circ f$. Let $\Rat_d$ denote the space of degree $d$ rational maps $\oP^1\to\oP^1$. The quotient of $\Rat_d$ by $PGL(2,\cx)$ can be viewed as the moduli space of abstract degree $d$ covers of $\oP^1$, but since the action of $PGL(2,\cx)$ has fixed points, this quotient is not manifold.
On the other hand, we can associate to $\phi$ its branch divisor. Classical Hurwitz conditions \cite{Hur} imply that, given an effective divisor $B$ of degree $2d-2$ on $\oP^1$, there exist, up to automorphisms, only finitely many rational maps $\phi:\oP^1\to \oP^1$ of degree $d$  with branch divisor $B$.  
Let $\phi\in \Rat_d$ and consider the induced sequence 
\begin{equation} 0\longrightarrow T\oP^1\stackrel{d\phi}{\longrightarrow} \phi^\ast T\oP^1\longrightarrow\sF_\phi\longrightarrow 0,\label{sF2}\end{equation}
where $\sF_\phi$ is supported on the ramification divisor of $\phi$. The space of global sections of the middle term is naturally isomorphic to $T_\phi\Rat_d$, while global sections of $T\oP^1$ correspond to infinitesimal automorphisms of $\oP^1$. Thus we can identify global sections of $\sF_\phi$ with deformations of the branch divisor $B$ of $\phi$, i.e.\ locally on $\Rat_d$ we have a natural isomorphism $H^0(\oP^1,\sF_\phi)\simeq T\oP^{2d-2}$.

\subsection{The geometry of the map $\rho$\label{map-rho}} With these preparations, we can prove:
%%%
\begin{proposition} The map $\rho$ is a submersion on an open subset of $X_{d,0}^{(0)}$ where $h^1({T_F}|_C)=0$. This open subset contains $X_{d,0}^{(2)}$.
\end{proposition}
\begin{proof} Let $C\in X_{d,0}^{(0)}$. We have an analogue of \eqref{sF2}: 
\begin{equation} 0\longrightarrow TC\stackrel{d\pi}{\longrightarrow} \pi^\ast T\oP^1\longrightarrow\sF\longrightarrow 0.\label{sF}\end{equation}
The sheaf $\sF$ is supported on the ramification divisor $R_C$ and is isomorphic to $\sF_\phi$ for any parameterisation of $C$. Owing to the above discussion we have a natural isomorphism $T_{B_C}\oP^{2d-2}\simeq H^0(C,\sF)$. 
\par
We also have the following two short exact sequences:
\begin{equation} 0\longrightarrow TC\longrightarrow TZ|_C{\longrightarrow} N_{C/Z}\longrightarrow 0,\label{normal}\end{equation}
\begin{equation} 0\longrightarrow T_F\longrightarrow TZ\stackrel{d\pi}{\longrightarrow} \pi^\ast T\oP^1\longrightarrow 0,\label{dpi}\end{equation}
where $T_F$ is the vertical tangent bundle.
Observe that the composition 
$$  TZ|_C\stackrel{d\pi}{\longrightarrow} \pi^\ast T\oP^1\longrightarrow \sF$$
factors through $N_{C/Z}$, and we obtain the following short exact sequence of sheaves on $C$:
\begin{equation} 0\longrightarrow {T_F}|_C\longrightarrow N_{C/Z}\longrightarrow \sF\longrightarrow 0.\label{NtoF}\end{equation}
The induced map $H^0(C,N_{C/Z})\to H^0(C,\sF)$ is $d\rho|_C$, and the first statement follows. If  $C\in X_{d,2}^{(2)}$, then $N_{C/Z}\simeq \sO_{\oP^1}(2d-1)^{\oplus 2}$. Sequence
\eqref{normal} implies then that the direct summands of $TX|_C$ have degree at most $2d-1$. Sequence \eqref{dpi}, restricted to $C$, implies now that the direct summands of ${T_F}|_C$ have degree at most $2d-1$. 
Since $c_1({T_F}|_C)=2d$, it follows that the direct summands of ${T_F}|_C$ have positive degree. 
\end{proof}
%%%%%%
%%%%%
We now consider the structure of the fibres of $\rho$. As discussed in \S\ref{covers},  the connected components of $\rho^{-1}(B)$ correspond to $PGL(2,\cx)$-orbits of rational maps with branch divisor $B$. Let us fix such a rational map $\phi:\oP^1\to\oP^1$, and suppose that there exists a $C_0\in X^{(0)}_{d,0}$ with a parameterisation $f_0:\oP^1\to C_0$ such that $\pi\circ f_0= \phi$. Then the connected component $X_\phi$  of $\rho^{-1}(B)$ containing $C$ is isomorphic to the space of embeddings $f:\oP^1\to Z$  such that $\pi\circ f=\phi$. Let $\phi^\ast Z$ denote the fibred product
\begin{equation} \phi^\ast Z=\{(t,z)\in \oP^1\times Z\:;\: \phi(t)=\pi(z)\},\label{fibred}
\end{equation}
and $\tilde\phi:\phi^\ast Z\to Z$ the projection on the second coordinate. 
We conclude that $X_\phi$ is isomorphic to the open subset of the Kodaira moduli space of sections $s$ of $\phi^\ast Z\to \oP^1$ such that $\tilde\phi\circ s$ is an isomorphism.
The tangent space to $X_\phi$ at $C$ is canonically isomorphic to $H^0(C, {T_F}|_C)$ (on the open subset where $h^1({T_F}|_C)=0$).
%%%%%
\begin{remark} Let $s:\oP^1\to Z$ be a section of $\pi$ with normal bundle isomorphic to $\sO(1)\oplus \sO(1)$. Then $s\circ \phi$ is a section of $\phi^\ast Z$ with normal bundle $\sO(d)\oplus\sO(d)$. Hence $\phi^\ast Z$ has a $(2d+2)$-dimensional smooth family of sections with this normal bundle. A generic element of this family will map to an embedded $\oP^1$ in $Z$. Sequence \eqref{NtoF} implies then that the normal bundle $N$ of this $\oP^1$ satisfies $h^1(N(-1))=0$. Consequently, each fibre of $\rho$ contains elements of $X_{d,0}^{(1)}$.\label{X1}
\end{remark}

\begin{example} Let $Z$ be the twistor space of the flat $\oR^4$, i.e. the total space of $\sO_{\oP^1}(1)\oplus \sO_{\oP^1}(1)$. Equivalently $Z=\oP^3\backslash \oP^1$, where $\oP^1=\{[z_0,z_1,0,0]\}$. The map $\pi:Z\to \oP^1$ is then the projection onto the last two coordinates. Let $C$ be a degree $d$ rational curve in $Z$, parameterized by $[f_0(u,v),\dots,f_3(u,v)]$, where $f_i(u,v)$ are homogeneous polynomials of degree $d$, $i=0,\dots,3$. The normal bundle of $C$ is then the cokernel of
$Df:\sO_{\oP^1}(1)\oplus\sO_{\oP^1}(1)\to \sO_{\oP^1}(d)^{\oplus 4}$, where $Df$ is the Jacobian matrix of $(f_0,\dots,f_3)$ \cite{GS}. The sheaf $\sF$ is the cokernel of $D\phi:\sO_{\oP^1}(1)\oplus\sO_{\oP^1}(1)\to \sO_{\oP^1}(d)\oplus \sO_{\oP^1}(d)\simeq {T_F}|_C$ where $\phi=(f_2,f_3)$, and 
$T_F=\{(a,b,0,0)\in TZ\}$. If $N_{C/Z}\simeq \sO_{\oP^1}(2d-1)\oplus \sO_{\oP^1}(2d-1)$, then we have an exact sequence
$$ 0\to \sO_{\oP^1}(1)\oplus\sO_{\oP^1}(1)\stackrel{Df}{\longrightarrow} \sO_{\oP^1}(d)^{\oplus 4}\stackrel{(\alpha_1,\alpha_2)}{\longrightarrow} \sO_{\oP^1}(2d-1)\oplus \sO_{\oP^1}(2d-1)\to 0,$$
where $\alpha_1$ and $\alpha_2$ are $2\times 2$ matrices of degree $d-1$ homogeneous polynomials in $u,v$. If we write 
$\phi=(f_2,f_3)$ and $\psi=(f_0,f_1)$, then the exactness of the above sequence implies $\alpha_1D\psi+\alpha_2D\phi=0$. The sequence \eqref{NtoF} is then 
$$  0\longrightarrow \sO_{\oP^1}(d)\oplus \sO_{\oP^1}(d) \stackrel{\alpha_1}{\longrightarrow} \sO_{\oP^1}(2d-1)\oplus \sO_{\oP^1}(2d-1)\longrightarrow \sF\longrightarrow 0.$$
The connected component $X_\phi$ of $\rho^{-1}(B)$ is an open subset of $\cx^{2d+2}$ consisting of pairs $(f_0(u,v),f_1(u,v))$ of homogeneous polynomials of degree $d$ such that $[f_0,f_1,f_2,f_3]$ is an embedding.\label{P3}
\end{example}

\section{Real manifolds\label{realm}}

We now suppose, in addition, that $Z$ is equipped with an antiholomorphic involution $\sigma$ covering the antipodal map on $\oP^1$. We denote by $M_d^{(i)}$, $i=0,1,2$, the $\sigma$-invariant part of the corresponding  $X_d^{(i)}$, and by  $M_{d,0}^{(i)}$ the $\sigma$-invariant part of  $X_{d,0}^{(i)}$. The manifolds $M_d^{(1)}$ and $M_d^{(2)}$ are equipped with the real versions of the geometry stated in Proposition \ref{geom}, i.e.\ an integrable quaternionic $2$-Kronecker structure in the case of $M_d^{(1)}$, and a hypercomplex or pseudo-hyperk\"ahler structure on $M_d^{(2)}$. In the case of rational curves we have the following restriction on $d$:
%%%%
\begin{lemma}{\cite[Prop.\ 5.9]{BP2}} Let $C$ be a connected projective curve of arithmetic genus $0$ equipped with a flat projection $\pi:C\to\oP^1$ of degree $d$. If $C$ admits an antiholomorphic involution covering the antipodal map on $\oP^1$,  then $d$ is odd.\hfill $\Box$\label{d-odd}\end{lemma}
%%%%
We assume, therefore, that $d$ is odd. The restriction of \eqref{rho} to $ M_{d,0}^{(0)}$ yields a smooth map
\begin{equation} \rho:  M_{d,0}^{(0)}\longrightarrow \oR\oP^{2d-2}.\label{rho2}\end{equation}
The connected components of its fibres correspond to $SO(3)$-orbits of rational maps $\phi:\oP^1\to \oP^1$ of degree $d$ which commute with the antipodal map (up to automorphisms). We recall the notion of a $d$-hypercomplex manifold \cite{DM1,DM2,BiTAMS}:
\begin{definition} Let $d\in \oN$ be odd. An almost $d$-hypercomplex structure on a smooth manifold $M$ is given by an isomorphism
 $T^\cx M\simeq E\otimes \cx^{d+1}$, where $E$ is a quaternionic vector bundle. Moreover, this isomorphism  is required to intertwine the complex conjugation on $T^\cx M$ and the tensor product of the quaternionic structure on $E$ and the standard quaternionic structure on $\cx^{d+1}$.
 \par
 An almost $d$-hypercomplex structure is integrable, i.e.\ a $d$-hypercomplex structure, if, for each Borel subgroup $B_\zeta\subset SL(2,\cx)$, $\zeta\in \oP^1$, the subbundle $E\otimes K_\zeta$ is involutive, where $K_\zeta$ is the direct sum of all, except the lowest, weight subspaces of $\cx^{d+1}$ for the standard irreducible representation of $SL(2,\cx)$. \label{d-hcx}
\end{definition}
As discussed in \cite{BiTAMS}, this is the natural geometry on the space of sections of a holomorphic submersion $\pi:Z\to \oP^1$, the normal bundle of which splits as $\bigoplus\sO(d)$.
%%%%
%%%%
\begin{proposition} The  fibres of the map $\rho$ restricted to the open subset of $  M_{d,0}^{(0)}$ where ${T_F}|_C\simeq \sO_{\oP^1}(d)\oplus\sO_{\oP^1}(d)$ 
have a natural $d$-hypercomplex structure.
\end{proposition}
\begin{proof} Let $M_\phi$ be a connected component of a fibre of $\rho$  determined by a rational map $\phi:\oP^1\to \oP^1$ of degree $d$, commuting with the antipodal map. The arguments of the previous section imply that $M_\phi$ is an open subset of the $\sigma$-invariant part of the Kodaira moduli space of sections of $\phi^\ast Z$. As observed in the previous section, the normal bundle of such a section corresponding to $C$ in the open subset of the statement is isomorphic to ${T_F}|_C\simeq \sO_{\oP^1}(d)\oplus\sO_{\oP^1}(d)$. 
\end{proof}
%%%%
 
\subsection{Signature of the metric}

Suppose now that $Z$ is also equipped with an $\sO(2)$-valued symplectic $2$-form along fibres of $\pi$, which is compatible with the real structure. Then the manifold $M_{d,0}^{(2)}$ has a natural pseudo-hyperk\"ahler metric $g$. We shall now use the description of the induced symplectic form $\omega^{[d]}$ on the bundle of $E$ over $M_{d,0}^{(2)}$, given in \S\ref{Douady}, to determine the signature of the metric. A real tangent vector in $TM_{d,0}^{(2)}\simeq E\otimes \cx^2$ can be written as $x=(e,-je)$, where  $j$ is the quaternionic structure of $E$, and the metric is then  (cf., e.g., \cite[(3.103)]{HKLR}) 
\begin{equation}g(x,x)=-2\omega^{[d]}(e,je).\label{metric}\end{equation} 
\par
Let now $C\in M_{d,0}^{(2)}$. We fix a parametrisation $r:\oP^1\to C$ such that $\sigma\circ r=r\circ \sigma$ (where $\sigma:\oP^1\to\oP^1$ denotes the antipodal map) and consider its composition with $\pi|_C$. This  is a degree $d$ rational map, commuting with the antipodal map. Without loss of generality we can assume that neither $0$ nor $\infty$ is mapped to $\infty$. We can then write this rational map as  $p(t)/q(t)$, where $p$ and $q$ are relatively prime degree $d$ polynomials in the affine coordinate $t$ on $\oP^1$.
%%%%
\par
Since $C\in M_{d,0}^{(2)}$, its normal bundle $N$ is isomorphic to $\sO_{\oP^1}(2d-1)\oplus \sO_{\oP^1}(2d-1)$. Then $N(-1)\simeq\sO_{\oP^1}(d-1)\oplus\sO_{\oP^1}(d-1)$, and the isomorphism 
$$E_C\otimes \cx^2\simeq  H^0(C,N(-1))\otimes \cx^2\longrightarrow H^0(C,N)\simeq T_C  M_{d,0}^{(2)}$$
can be written as
$$ H^0(\oP^1,\sO(d-1)\oplus\sO(d-1))\otimes\cx^2 \ni \bigl((f_1,g_1),(f_2,g_2)\bigr)\longmapsto (pf_1+qg_1,pf_2 +qg_2).$$
We can also assume that the quaternionic structure of $E_C$ is the standard one on $ H^0(\oP^1,\sO(d-1)\oplus\sO(d-1))$, i.e.
$$ j\bigl(f(t),g(t)\bigr)=t^{d-1}\Bigl(-\ol{g(-1/\bar t\,)},\ol{f(-1/\bar t\,)}\Bigr).$$
We now unravel the description of $\omega^{[d]}$, given in the proof of Proposition \ref{other}.
Let $(f_i,g_i)\in H^0(C,N(-1))$, $i=1,2$, be two sections, consisting of pairs of polynomials of degree $d-1$. Then $f_1g_2-g_1f_2\in H^0\bigl(C,(\Lambda^2N)(-2)\bigr)$, which we view (using $\omega\in H^0(C,K_Z(4)|_C)$) as a section of $H^0(C,K_C(2))$. The corresponding meromorphic $1$-form is $(f_1g_2-g_1f_2)(q^2(t))^{-1}dt$, and it has  poles bounded by $2(q(t))$. The  extension class in $H^1(\oP^1,\sO_{\oP^1}(-2))$ can be viewed as the Laurent tail $\zeta^{-1}\cdot \infty$, and its pullback is then 
$$\sum_{i=1}^d \bigl( \text{linear term of $q(t)/p(t)$ at $t=t_i$} \bigr)\cdot t_i,$$
where $t_i$ are the zeros of $q$ (since $\zeta=p(t)/q(t)$). The pairing of $H^0(C,K_C(2))$ and $H^1(C,\sO(-2))$ is given by the residue map
$$\sum_{i=1}^d \Res_{t=t_i} \frac{q(t)}{p(t)}(f_1g_2-g_1f_2)(q^2(t))^{-1}dt=\sum_{i=1}^d \Res_{t=t_i} \frac{f_1g_2-g_1f_2}{p(t)q(t)}dt.$$
Let us write, for a polynomial of degree $k$,
$$ \tau (f)(t)=(-t)^k\ol{f\bigl(-1/\bar t\bigr)}.$$
The square of this map is ${\rm Id}$ if $k$ is even, and $-{\rm Id}$ if $k$ is odd. The fact that $p/q$ commutes with the antipodal map means that $p=-\tau(q)$.
\par
Let $\gamma$ be a simple contour in $\cx$ separating the roots from $q$ from the roots of $\tau(q)=-p$. It follows from the above and from \eqref{metric} that the metric on $T_C M_{d,0}^{(2)}$ is equal to
\begin{equation} \|x\|^2=\bigl\|\bigl((f,g),-j(f,g)\bigr)\bigr\|^2=\frac{1}{\pi i}\oint_\gamma \frac{f\tau(f)+g\tau(g)}{q\tau(q)}dt.\label{hermitian}\end{equation}
%%%%
We want to determine the signature of the right-hand side on pairs $(f,g)$ of polynomials of degree $d-1$.  Owing to the continuity, it is enough to compute the signature for one particular $q$, say $q(t)=t^d$. Then $\tau(q)=1$ (since $d$ is odd) and the right-hand side is the middle degree term of $2f\tau(f)+2g\tau(g)$, i.e.:
$$ 2\sum_{i=0}^{d-1}(-1)^{d-1-i}|f_i|^2+2\sum_{i=0}^{d-1}(-1)^{d-1-i}|g_i|^2.$$
Therefore the  signature of the metric $g$ is $(2d+2,2d-2)$.

\section{Example: gravitational instantons of type $A_k$}
Consider an ALE or an ALF gravitational instanton $M$ of type $A_k$. We recall, after Hitchin \cite{Hit-pol}, its construction using twistor methods. The twistor space of $M$ has a singular model given as a hypersurface $\bar Z$ in the total space of a vector bundle over $T\oP^1$. If $\zeta$ is the affine coordinate on $\oP^1$ and $\eta$ is the corresponding fibre coordinate on $T\oP^1$, we denote by $L^c$, $c\in \cx$, the line bundle on $T\oP^1$ with transition function $\exp(-c\eta/\zeta)$ from $\zeta\neq\infty$ to $\zeta\neq 0$. Then $\bar Z$ is given by
\begin{equation*} \{(x,y,z)\in L^c(k)\oplus L^{-c}(k)\oplus \sO_{\oP^1}(2)\,;\, xy=\prod_{i=1}^k(z-a_i(\zeta))\}\label{hypers}\end{equation*}
where $c$ is real and $a_i$ are quadratic polynomials satisfying reality conditions. $M$ is then the space of real sections of $\pi:\bar Z\to \oP^1$ obtained by choosing an arbitrary real section $z(\zeta)=(x_2+ix_3)+2x_1\zeta-(x_2-ix_3)\zeta^2$ of $\sO_{\oP^1}(2)$,  and dividing the set of all zeros of $z(\zeta)-a_i(\zeta)$, $i=1,\dots,k$, into two subsets $\Delta_1$, $\Delta_2$, interchanged by the antipodal map. This can be done consistently as shown in \cite{Hit-pol}. The sections of $\pi:\bar Z\to \oP^1$ are then
\begin{equation} x(\zeta)=Ae^{c(x_1-(x_2-ix_3)\zeta)}\prod_{\zeta_i\in\Delta_1}(\zeta-\zeta_i),\quad y(\zeta)=Be^{-
c(x_1-(x_2-ix_3)\zeta)
%c((x_2+ix_3)/\zeta+x_1)
}\prod_{\zeta_i\in\Delta_2}(\zeta-\zeta_i),
\label{xy}\end{equation}
over $\zeta\neq \infty$.
The nonzero scalars $A,B$ are determined up to a circle action, which yields an isometric $S^1$-action on $M$.
We remark that resolving the singularities of $\bar Z$ is not necessary for computing $M$ and its metric.
\par
We now discuss the geometry of real (i.e.\ $\sigma$-invariant) $\oP^1$'s of degree $d$, $d$ - odd, in $\bar Z$. 
Let $\phi=p(t)/q(t)$ be a  rational map of degree $d$, commuting with the antipodal map, and assume for simplicity that $t=\infty$ is not a pole of $\phi$. 
The function $\phi$ can be viewed as the transition function for the bundle $\sO_{\oP^1}(d)$ from $U_0=\{q\neq 0\}$ to $U_1=\{p\neq 0\}$. A section of $\sO_{\oP^1}(kd)$ is then represented by $b/q^k$ on $U_0$ and $b/p^k$ on $U_1$, where $b$ is a polynomial of degree $kd$. Let $z=b/q^2$ be a section of $\sO_{\oP^1}(2d)$ and write $b=b_0p+b_1q$. We get a section of the line bundle $L_\phi^c$ with transition function $\exp(cz/\phi)$  by setting
$$(s_0,s_1)=\bigl(\exp(-cb_0/q),\exp(cb_1/p)\bigr)$$
in $U_0$ and $U_1$ respectively. If we now consider the fibred product $\phi^\ast \bar Z$, as in \S\ref{map-rho}, then its sections, and hence the fibre of the map $\rho: M_{d,0}^{(2)}\to \oR\oP^{2d-2}$, are obtained in the same way as for $d=1$: choose an arbitrary real section $z(t)=b(t)/q(t)^2$ of $\sO_{\oP^1}(2d)$, divide the zeros of all $z(t)-a_i(\phi(t))$ into two sets, and obtain $x(t), y(t)$ as in \eqref{xy}, replacing the exponential factors by $\bigl(\exp(-cb_0/q)$ over $q(t)\neq 0$ and by $\exp(cb_1/p)$ over $p(t)\neq 0$. The space of real sections of $\phi^\ast\bar Z$ with normal bundle $\sO(d)\oplus\sO(d)$ is nonempty owing to Remark \ref{X1}, and is a $d$-hypercomplex analogue of the original gravitational instanton, as introduced in \cite[\S 3.1.2]{DM2}. 
\par
A generic section $s$ of $\phi^\ast\bar Z$ will yield an embedded $\oP^1$ in $\bar Z$, and hence,
by varying $\phi$, we obtain a $4d$-dimensional space of embedded real  $\oP^1$'s of degree $d$ in $\bar Z$. 
\par
We claim that, for  generic $c$ and $a_i$,  $i=1,\dots,k$, the normal bundle of a generic such curve is $\sO(2d-1)\oplus \sO(2d-1)$, i.e.\ $M_{d,0}^{(2)}$ is nonempty (and hence of dimension $4d$). Indeed, were this not the case, the normal bundle of every degree $d$ rational curve (flat over $\oP^1$) in the twistor space of $\bigl(\oR^4\backslash \{0\}\bigr)/\oZ_k$ would also be different from $\sO(2d-1)\oplus \sO(2d-1)$. This twistor space $Z_0$ is the quotient by $\oZ_k$ of the total space $W$ of $\sO(1)\oplus\sO(1)$ with the zero section removed. In particular a generic degree $d$ rational curve in $W$ descends to a rational curve of degree $d$ curve in $Z_0$ with isomorphic normal bundle. Since $W$ is an open subset of $\oP^3$, a generic degree $d$ $\oP^1$ in $W$ has normal bundle isomorphic to  $\sO(2d-1)\oplus \sO(2d-1)$ \cite{GS}. This contradiction proves our claim.
\par
We can say more in the case  $c=0$, i.e.\  when $Z$ is the twistor space of an ALE manifold.
%%%%
The fibred product  $\phi^\ast \bar Z$ is then a hypersurface in the total space of the vector bundle $E_d=\sO(kd)\oplus\sO(kd)\oplus\sO(2d)$.
If $s$ is a section of $\phi^\ast \bar Z$,  given by homogeneous polynomials $x(u,v),y(u,v),z(u,v)$ of degrees $kd,kd,$ and $2d$, then its normal bundle fits into a short exact sequence
\begin{equation} 0\longrightarrow N_{s/\phi^\ast\bar Z} \stackrel{j}{\longrightarrow} E_d\longrightarrow \sO(2kd) \longrightarrow 0,\label{NCphi}\end{equation}
since $N_{s/E_d}\simeq E_d$. The projection $ E_d\longrightarrow \sO(2kd)$ is given by $$\Bigl[y,x,-\sum_i\prod_{j\neq i}(z-a_j)\Bigr]^T,$$
from which one can compute $N_{s/\phi^\ast\bar Z} $.

The original twistor space $\bar Z$ is a hypersurface in the total space of the bundle $E_1$. We can view the curve $C$ given by $\phi $ and $s$ as being embedded in $E_1$. Its normal bundle fits then into a short exact sequence (cf.\  Ex.\ \ref{P3}):
\begin{equation} 0\longrightarrow \sO(1)\oplus\sO(1) \stackrel{\Psi}{ \longrightarrow} \sO(kd)^{\oplus 2}\oplus\sO(2d)\oplus\sO(d)^{\oplus 2} \longrightarrow N_{C/E_1} \longrightarrow 0,\label{NE1}\end{equation}
\nopagebreak[9]
where $\Psi$ is the Jacobi matrix of $[x(u,v),y(u,v),z(u,v),p(u,v),q(u,v)]^T$. We can extend \eqref{NCphi} and \eqref{NE1} to a commutative diagram:
$$
\begin{tikzcd} & & 0 \ar[d]& 0\ar[d] &\\[-6pt] & & N_{s/\phi^\ast\bar Z}\oplus\sO(d)^{\oplus 2}\ar[r,"\nu"] \ar[d, "j\oplus {\rm Id}"] & N_{C/\bar Z}\ar{r}\ar{d} & 0\\
0 \ar[r] & \sO(1)\oplus\sO(1)\ar[r, "\Psi"] \ar[ru, "\lambda", dashrightarrow] & E_d\oplus\sO(d)^{\oplus 2}\ar[r] \ar[d] &N_{C/E_1} \ar[r]\ar[d] & 0\\ & & \sO(2kd)\ar[d]\ar[r, equal] &\sO(2kd)\ar[d] &\\[-6pt] & & 0 & 0 &
\end{tikzcd}
$$
It is now the matter of (a complicated) linear algebra to compute the map $\lambda$ from $\Psi$ and from the vertical projection $E_d\to \sO(2kd)$. Assuming that $ N_{s/\phi^\ast\bar Z}\simeq \sO(d)\oplus\sO(d)$, $\lambda$ will be a $4\times 2$ matrix of degree $d-1$ polynomials in $u,v$, the coefficients of which depend on the polynomials $x,y,z,p,q$. The map $\nu$ can then be computed from $\lambda$, and the condition that $N_{C/\bar Z}\simeq \sO(2d-1)\oplus\sO(2d-1)$ can be written as a determinant of a matrix given by coefficients of $\lambda$. The polynomials $x$ and $y$ depend (up to scale) algebraically on $z,p,q$ and, hence, $X_{d,0}^{(2)}$, and its real part $M_{d,0}^{(2)}$, are described by this algebraic relation between  the coefficients of arbitrary polynomials $z,p,q$ of degrees $2d,d,d$. 
%%%%
\begin{remark}
In principle, once the maps $\lambda$ and $\nu$ are determined, the pseudo-hyperk\"ahler metric on $M_{d,0}^{(2)}$ can be computed using the method of the previous section.
\end{remark}
%In the case $c\neq 0$, i.e.\ ALF, we can conclude, using deformation theory, that at least for small enough $c$, a generic $C$ has the correct normal bundle, equal to $ \sO(2d-1)\oplus\sO(2d-1)$. This argument has been used in \cite[p.\ 394]{Besse} in the case when $d=1$. Of course we expect that the generic normal bundles remains correct for all $c$.
\begin{remark} We should like to point out that in the ALE case, the singular model $\bar Z$ compactifies to a hypersurface in the weighted projective space $\oP=\oP(1,1,2,k,k)$. For more details on this compactification see \cite{Kron}. Since the degree of the hypersurface is $2k$, the adjunction formula implies that its canonical sheaf is isomorphic to $\sO_{\oP}(2k-1-1-2-k-k)\simeq \sO_{\oP}(-4)$, and hence the compactification of $\bar Z $ is Fano.
\label{Fano}\end{remark}

\end{document}